\documentclass[11pt]{article}

\usepackage{cite}
\usepackage{setspace}
\usepackage[vcentering,dvips]{geometry}
\usepackage{amsmath, amsthm, amssymb, xspace}
\theoremstyle{plain} % default
\newtheorem{theorem}{Theorem}
\newtheorem{lemma}[theorem]{Lemma}
\newtheorem{corollary}[theorem]{Corollary}

\theoremstyle{definition}

\theoremstyle{remark}

\usepackage[usenames]{color}

\newcommand{\name}[1]{\textsf{#1}}
\newcommand{\sss}{\name{SUBSET-SUM}\xspace}
\newcommand{\bR}{\mathbb{R}}
\newcommand{\eps}{\varepsilon}
\newcommand{\abs}[1]{\left\lvert{#1}\right\rvert}
\newcommand{\norm}[1]{\left\lVert{#1}\right\rVert}
\newcommand{\normsq}[1]{\norm{#1}_2^2}

\DeclareMathOperator*{\argmin}{argmin}

\begin{document}

\newcommand{\email}[1]{e-mail: \texttt{#1}}
\title{On the complexity of Mumford-Shah type regularization, viewed as a 
relaxed sparsity constraint }
\author{Boris~Alexeev, and Rachel~Ward,% <-this % stops a space
\thanks{B. Alexeev is with the Department of Mathematics at Princeton
  University, Fine Hall, Washington Road, Princeton, NJ 08544 USA
  \email{balexeev@math.princeton.edu}.}%
\thanks{R. Ward is with the Department of Mathematics at the Courant Institute
of Mathematical Sciences, New York University, 251 Mercer St,
New York, NY 10012 USA \email{rward@cims.nyu.edu}.}}

\maketitle

\begin{abstract}
We show that inverse problems with a \emph{truncated quadratic regularization} are
NP-hard in general to solve, or even approximate up to an additive error.  This stands in contrast to 
the case corresponding to a finite-dimensional approximation to the Mumford-Shah functional, where the operator involved is the identity and for which 
polynomial-time solutions are known. Consequently, we confirm the infeasibility of 
any natural extension of the Mumford-Shah functional to general inverse problems.  A connection 
between truncated quadratic minimization and sparsity-constrained minimization is also discussed. 

\vspace{1mm}
\noindent {\bf Keywords:} inverse problems, Mumford-Shah functional, truncated quadratic regularization, sparse recovery, NP-hard, thresholding, SUBSET-SUM

\end{abstract}

\section{Introduction}
Consider a discrete signal $x \in \mathbb{R}^N$ sampled from
a piecewise smooth signal and revealed through measurements $y = Ax +
e$, where $e \in
\mathbb{R}^m$ is observation noise and $A: \mathbb{R}^N \mapsto \mathbb{R}^m$ 
is a known linear operator identified with an $m \times N$ real matrix (representing, for
instance, a blurring or partial obscuring of the data). 
Consider the
\emph{truncated quadratic minimization problem}, 
\begin{align}
\hat{x} &= \argmin_{x \in \mathbb{R}^N} \mathcal{J}(x), \nonumber \\
\mathcal{J}(x) &= \normsq{ Ax - y } + \sum_{j=1}^{N-1} Q(x_{j+1} - x_j),
\label{eq:opt}
\end{align}
with truncated quadratic penalty term $Q(u) = \alpha \min\{u^2,
\beta\}$ parametrized by $\alpha, \beta > 0$. 
Since its
introduction in 1984 by Geman and Geman in the context of 
image restoration \cite{geman, blake_zisserman, nikolova}, this problem has been the
subject of considerable theoretical and practical interest,  finding applications 
ranging from visual analysis to crack detection in fracture mechanics \cite{rondi1, rondi2}.
The choice of regularization is motivated as follows: $Q$ desires to smooth small differences $\abs{ x_{j+1} - x_j} \leq \sqrt{\beta}$ where it acts
quadratically, but suspends smoothing over larger differences.
\\
\\
\noindent From a statistical point of view, the quadratic data-fidelity term $\normsq{ Ax - y }$
can be viewed as a log-likelihood of the data under the hypothesis that $e$ is Gaussian random noise,
while the truncated quadratic regularization term corresponds to the energy of a piecewise
Gaussian Markov random field \cite{geman, besag, jeng_woods}.
\\
\\
\noindent The truncated quadratic minimization problem is non-smooth
and highly non-convex. However, several characterizations of the
minimizers have been unveiled \cite{nikolova, fornasier_ward}.  It is
known for instance that minimizers exist and
satisfy a ``gap" property \cite{nikolova, fornasier_ward}:
the magnitude of successive differences of such
solutions are either smaller than a first threshold or larger than a
second, \emph{strictly} larger threshold. These thresholds are
independent of the observed data $y = Ax + e$ and depend only on the
regularization parameters $\alpha$ and $\beta$.  This dependence is
explicit, so that a priori information about the thresholds can be
incorporated into choice of regularization parameters.
\\
\\
\noindent  When $A$ is the $N \times N$ identity matrix, the truncated quadratic objective function can be viewed as a discretization of the \emph{Mumford-Shah functional}\footnote{In \cite{chambolle1}, the minimizers $\hat{x} = \hat{x}_{(N)}$ of the $N$-dimensional truncated quadratic minimization problem with $A = I$, parameters $(\alpha_{(N)}, \beta_{(N)}) = (N^2 \alpha, N \beta)$, and $y_{(N)} = y_{(j/N)}$ identified with discrete samples from a continuous function $y \in L^{\infty}[0,1]$, were shown to converge to the minimizer of the Mumford-Shah functional, 
\begin{eqnarray}
\label{MS}
\hat{x} &=& \argmin_{x \in SBV[0,1]} {\cal F}(x), \nonumber \\
{\cal F}(x) &=& \int_{[0,1] \setminus S_x} \Big( (x - y)^2 + \alpha \normsq{ \nabla x } \Big) dx + \alpha \beta \abs{ S_x }, \nonumber
\end{eqnarray}
over the space $SBV$ of bounded variation functions on $[0,1]$ with vanishing Cantor part.  Note that SBV functions have a well-defined \emph{discontinuity set} $S_x$ of finite cardinality $\abs{ S_x }$; see \cite{fornasier_ward} for more details.}, which motivated the variational approach for edge detection and image segmentation with its introduction in 1988.  When $A$ is the identity matrix as such,  the truncated quadratic minimization problem can be solved in polynomial-time using dynamic programming \cite{chambolle1}.
However, for general $m \times N$ matrices $A$, existing algorithms for minimizing the functional \eqref{eq:opt} guarantee convergence to \emph{local} minimizers at best \cite{fornasier_ward}.
\\
\\
\noindent  In this paper, we show that the truncated quadratic minimization problem is NP-hard in general, certifying that the present convergence guarantees are the best one could hope for.  Consequently, the Mumford-Shah functional \eqref{MS} cannot be tractably extended to general inverse problems.

\section{Truncated quadratic minimization reformulated}

It will be helpful to recast the truncated quadratic minimization problem
 in terms of the discrete differences $u_j = x_{j+1}
- x_j$, effectively decoupling the action of the regularization term $Q$.
We may express this change of variables in matrix notation as $u =
D x$, with $D: \mathbb{R}^N \rightarrow \mathbb{R}^{N-1}$ the
\emph{discrete difference} matrix,
\[
D = \left ( \begin{array}{cccccc} -1 & 1 & 0& \dots & \dots &0\\
0&-1&1&0 &\dots&0\\
\dots\\
0&0&\dots&\dots&-1&1
\end{array} \right ).
\]
The null space of $D$, which we denote in the following by $\cal{N}(D)$,
is simply the one-dimensional subspace of constant vectors in
$\mathbb{R}^N$.  The orthogonal projection of a vector $x \in
\mathbb{R}^N$ onto this subspace is the constant vector $c$ whose
entries coincide with the mean value $\frac{1}{N}\sum_{j =1}^N x_j$ of
$x$, while its projection onto the orthogonal
complement of $\cal{N}(D)$ is given by the
least squares solution $D^{\dagger} D x$, where $D^\dagger$ is the
pseudo-inverse matrix of $D$ in the Moore-Penrose sense.   
These observations yield the orthogonal decomposition
$x = D^\dagger D x + c$, or, incorporating the substitution $u = Dx$,
 the decomposition $x = D^\dagger u + c$.
\\
\\
\noindent  Minimization of $\mathcal{J}$, recast in terms of the variables
 $u$ and $c$, becomes
\begin{align}
 (\hat{c}, \hat{u}) &= \argmin_{c \in N(D),\; u \in \mathbb{R}^{N-1} } \mathcal{J}(c,u), \nonumber \\
 \mathcal{J}(c,u) &=  \normsq{ A D^\dagger u + Ac - y } + \sum_{j=1}^{N-1} Q(u_j),
\label{eq:recast}
\end{align}
where the primal minimizer $\hat{x}$ and $(c,u)$-minimizer $(\hat{c}, \hat{u})$ are interchangeable according to
$ \hat{x} = D^\dagger \hat{u} + \hat{c} $.
\\
\\
\noindent If the null space of $A$ contains the constant vectors, such as if $A = TD$ for an
 $m \times (N-1)$ matrix $T$, the minimization problem~\eqref{eq:recast} reduces to a function of $u$ only,
\[
\hat{u} = \argmin_{u \in \mathbb{R}^{N-1}}  \normsq{ A D^\dagger u - y } + \sum_{j=1}^{N-1} Q(u_j).
\]
Making the substitution $A = TD$ and using that $D D^\dagger = I$,
 we see in particular that any optimization problem of the
form $\hat{u} = \argmin_{u \in \mathbb{R}^{N-1}} \normsq{ T u - y } +
\sum_{j=1}^{N-1} Q(u_j)$ can be identified with an instance of a
truncated quadratic minimization problem~\eqref{eq:opt}. To summarize,

\begin{lemma}
Let $T: \mathbb{R}^{N-1} \mapsto \mathbb{R}^m$ be a linear operator
identified with a matrix of $\mathbb{R}^{m \times (N-1)}$.  The
minimization problem
\begin{equation}
\hat{u} = \argmin_{u \in \mathbb{R}^{N-1}} \normsq{ T u - y } + \sum_{j=1}^{N-1} Q(u_j)
\label{eq:6}
\end{equation}
corresponds to a truncated quadratic minimization problem
\begin{equation}
\hat{x} = \argmin_{x \in \mathbb{R}^N} \normsq{ TD x - y } + \sum_{j=1}^{N-1} Q(x_{j+1} - x_j)
\label{eq:7}
\end{equation}
in the sense that $\hat{x} = D^{\dagger} \hat{u}$ is a minimizer
for~\eqref{eq:7} if $\hat{u}$ minimizes~\eqref{eq:6}, while $\hat{u} =
D\hat{x}$ is a minimizer for~\eqref{eq:6} if $\hat{x}$ is a minimizer
for~\eqref{eq:7}.
 \label{lem:1}

\end{lemma}

\section{Reduction to \sss}

Recall that the complexity class NP consists of all problems whose
solution can be verified in polynomial time given a
certificate for the answer.  For example, the problem $\sss$ is
to determine, given nonzero integers $a_1, \dotsc, a_k$ and $C$, whether
or not there exists a subset $S$ of $\{1, \dotsc, k\}$ such that
$\sum_{i\in S} a _i = C$.  This problem is in NP because given any
particular subset~$S$, we can easily check whether or not its
corresponding sum is zero.
\\
\\
\noindent  Further recall that a \emph{polynomial-time many-one reduction} from a
problem $A$ to a problem $B$ is an algorithm that transforms an instance
of~$A$ to an instance of~$B$ with the same answer in time polynomial
with respect to the number of bits used to represent the instance
of~$A$.  Intuitively, this captures the notion that $A$ is no harder
than $B$, up to polynomial factors, and accordingly one may write $A \le
B$.  Finally, a problem $B$ is called NP-hard if every problem in NP is
reducible to $B$.  (Note that $A\le B$ and $B \le C$ imply $A\le C$, so
if an NP-hard problem $B$ reduces to a problem $C$, then $C$ is NP-hard
as well.)  NP-hard problems can not be solved in polynomial time unless
P$=$NP.
\\
\\
\noindent  In order to prove our NP-hardness result, we show that the known 
NP-hard problem $\sss$ admits a 
polynomial-time reduction to an instance of the truncated quadratic minimization problem.
Moreover, we show that any algorithm that could
efficiently \emph{approximate} this minimum (to within an additive error) 
could solve $\sss$
efficiently as well; that is, the search for even an approximate
solution to a truncated quadratic minimization problem is NP-hard as well.

\begin{theorem}
  Let $a_1, \dotsc, a_k$ and $C$ be given nonzero integers.  Then
  there exists a subset $S$ of $\{1, \dotsc, k\}$ such that $\sum_{i\in
    S} a_i = C$ if and only if $\min_{x\in \bR^{2k}} f(x) \le k$, where
  \begin{align*}
  f(x) = \left( C - \sum_{i=1}^k a_i x_i
  \right)^2 & + P\cdot \sum_{i=1}^k \left( 1 - x_i - x_{i+k} \right)^2
  \\ & +
  \sum_{i=1}^{2k} \min \left( 1, \frac{x_i^2}{\eps^2} \right),
  \end{align*}
with $0
  < \eps\le \frac{1}{4(\sum_i \abs{a_i})}$ and $P\ge \frac{2k}{\eps^2}$.
  Moreover, this minimum is never strictly between $k$ and $k+\frac 14$.
  
  \label{thm:2}
\end{theorem}
\begin{proof}
  Call a subset $S\subseteq \{1, \dotsc, k\}$ \emph{good} if $\sum_{i\in
    S} a_i = C$.  If a good subset $S$ exists, we may set
  \[
  x_i = \begin{cases} 1 & \text{if $i\in S$,} \\
    0 & \text{if $i\notin S$,}
  \end{cases}
  \qquad \text{and} \qquad
  x_{i+k} = \begin{cases} 0 & \text{if $i\in S$,} \\
    1 & \text{if $i\notin S$,}
  \end{cases}
  \]
  for $1 \le i\le k$.  Then $f(x) = k$ because the first and second
  terms vanish, and there are exactly $k$ nonzero $x_i$s.  Therefore, if
  a good subset exists, the minimum is at most $k$.
\\
\\
\noindent Suppose no good subset exists, yet there exists $x$ such that $f(x) <
  k+\frac 14$.  Consider the $k$ pairs of coordinates $x_i, x_{i+k}$ for
  $1\le i\le k$.  If both $\abs{x_i}$ and $\abs{x_{i+k}}$ were less than
  $\eps$, a single summand in the second term already exceeds $2k$, as
  $P\cdot ( 1 - x_i - x_{i+k} )^2 \ge P\cdot (1 - 2\eps)^2 \ge 2k$.
  Therefore, at least one of $\abs{x_i}$ and $\abs{x_{i+k}}$ exceeds
  $\eps$ and the third term is already at least~$k$.  If more than one
  of $\abs{x_i}$ and $\abs{x_{i+k}}$ exceeded $\eps$, then the third
  term would be at least $k+1$, so exactly one of the coordinates in
  each pair exceeds $\eps$ in absolute value.  If $\abs{x_i} \le \eps$,
  then $\abs{x_{i+k} - 1} \le 2\eps$ as otherwise $P\cdot (1 - x_i -
  x_{i+k})^2 \ge P\cdot (\eps)^2 \ge 2k$; the symmetric holds if
  $\abs{x_{i+k}} \le \eps$, so all of the $x_i$ are within $2\eps$ of
  either $0$ or $1$.
\\
\\
\noindent  Let $\overline{x}_i$ be the closer of $0$ and $1$ to $x_i$.  Then
  because no good subset exists and $C$ and the $a_i$ are integers,
  $\abs{C - \sum_{i=1}^k a_i \overline{x}_i} \ge 1$.  It follows that
  the first term $\left( C - \sum_{i=1}^k a_i x_i \right)^2 \ge (1 -
  \sum 2\eps a_i)^2 \ge (\frac12)^2 = \frac14$.  But the third term was
  already at least $k$, so this is a contradiction.  Therefore, if no
  good subset exists, we must have $\min_x f(x) \ge k+\frac 14$.
\end{proof}

\begin{corollary}
  Solving the truncated quadratic regularization problem, even to within an
  additive error, is NP-hard.
\end{corollary}
\begin{proof}
In light of Lemma \ref{lem:1}, minimization of the function $f$ is a truncated quadratic minimization problem, with $m=k+1$ and $N=2k$.  Therefore, we have reduced 
 the known NP-hard problem \sss to a truncated quadratic minimization problem. 
It remains to verify that this reduction is polynomial-time.  To see this, note that
Theorem \ref{thm:2} ensures that the minimum of $f$ is either at most $k$ or at
 least $k+\frac14$; thus, we only need to approximate each of the polynomially-many 
 entries in the matrices and vectors in $f$ to within a number of bits that is polynomial compared to the 
 number of bits needed to represent $\sum_i \abs{a_i}$.   
 \end{proof}

\section{Connection to sparse recovery}
% The \penalty0 in the following text prevents a line underflow -- it's
% debatable whether it looks any better this way
The only properties of the quadratic regularization term $\min\{
1,\penalty0 {x_i^2}/{\eps^2} \}$
needed for Theorem \ref{thm:2} were that it be bounded between $0$ and
$1$, equal to $0$ if $x_i = 0$, and equal to $1$ if $\abs{x_i} \ge
\eps$.  Indeed, Theorem \ref{thm:2} holds for \emph{any} regularization
term satisfying these properties; for example, one could consider hard
thresholding,
    \[ \abs{ x }_0 = \begin{cases} 0 & x = 0, \\ 1 & x \ne 0, 
    \end{cases} \]
    which generates the $\ell_0$ ``counting norm" $\norm{ x}_0 = \sum_{j=1}^N \abs{ x_j }_0$.  We then reprove the known result that the $\ell_0$-regularized optimization problem, 
   \begin{equation}
\hat{u} = \argmin_{u \in \mathbb{R}^{N}} \normsq{ T u - y } + \gamma \norm{u}_0,
\nonumber
\end{equation}
is NP-hard in general.  This functional is of considerable interest in the emerging area of sparse recovery, as it is guaranteed to produce sparse solutions for sufficiently large $\gamma$ and over a certain class of matrices \footnote{See \cite{CS} for matrix constructions
   that admit polynomial-time recovery algorithms for the $\ell_0$-regularized optimization problem.  All constructions at present involve an
   element of randomness, and a complete characterization of such matrices forms the core of the area known as \emph{compressed sensing}.}.  In this light, the truncated quadratic minimization problem may be 
   interpreted as a relaxation of the $\ell_0$-regularized optimization problem, and 
   our main result as showing that even such \emph{relaxations} of the $\ell_0$-regularized functional are NP-hard.

\bibliography{complexity}

\end{document}